\documentclass[12pt]{article}
\usepackage{amsthm}
\usepackage{times}
\usepackage{bbding}
\usepackage{authblk}
\usepackage[colorlinks]{hyperref}
\usepackage[margin=1in]{geometry}
\usepackage[utf8]{inputenc}
\usepackage{amsmath,accents}
\usepackage[normalem]{ulem}
\usepackage{mathtools}
\usepackage{amsfonts}
\usepackage{graphicx}
\usepackage{amssymb}
\usepackage{xcolor}

\newtheorem{theorem}{Theorem}[section]
\newtheorem{remark}[theorem]{Remark}
\newtheorem{lemma}[theorem]{Lemma}
\newtheorem{corollary}[theorem]{Corollary}
\newtheorem{proposition}[theorem]{Proposition}
\newtheorem{definition}[theorem]{Definition}

\newcommand{\spec}{\mathrm{spec}}
\newcommand{\ums}{\mathcal{U}}
\newcommand{\ms}{\mathcal{M}}
\newcommand{\dgh}{d_{\mathrm{GH}}}
\newcommand{\ugh}{u_{\mathrm{GH}}}
\newcommand{\eps}{\varepsilon}
\newcommand{\dH}{d_\mathrm{H}}

\title{A novel construction of Urysohn universal ultrametric space via the Gromov-Hausdorff ultrametric}

\author{Zhengchao Wan}

\affil{Department of Mathematics\\
 	The Ohio State University\\
 	Columbus, Ohio 43210\\
 	\href{mailto:wan.252@osu.edu}{wan.252@osu.edu}}
 	
\providecommand{\keywords}[1]
{
  \small	
  \textbf{\textit{Keywords---}} #1
}

\providecommand{\msc}[1]
{
  \small	
  \textbf{\textit{MSC---}} #1
}

\date{}

\begin{document}
\maketitle 

\begin{abstract}
We establish universality and ultra-homogeneity of $(\ums,\ugh)$, the collection of all compact ultrametric spaces endowed with the so-called \emph{Gromov-Hausdorff ultrametric}. This result also gives rise to a novel construction of the so-called \emph{$R$-Uryoshn universal ultrametric space} for each countable subset $R\subset\mathbb{R}_{\geq 0}$ containing $0$.
\end{abstract}

{\msc{51F99, 54E35}}

\keywords{Gromov-Hausdorff ultrametric,  universal ultrametric space, dendrograms}

\section{Introduction}
A metric space is called \emph{universal} if it contains isometric copies of all separable metric spaces. The study of universal metric space dates back to \cite{urysohn1927espace} in which Urysohn identified a unique (up to isometry) Polish\footnote{A metric space is called Polish if it is complete and separable.} universal space $\mathbb{U}$ (named the \emph{Urysohn universal metric space}) that satisfies the following \emph{ultra-homogeneity} condition: given a finite metric space $B$, a subset $A\subset B$ and an isometric embedding $\varphi:A\rightarrow \mathbb{U}$, there exists an isometric embedding $\psi:B\rightarrow\mathbb{U}$ such that $\psi|_A=\varphi$. 

An ultrametric space is a special metric space that satisfies the strong triangle inequality (cf. Equation \ref{eq:strong-triangle}). Vestfrid constructed in \cite{vestfrid1994universal} the first example of a universal and ultra-homogeneous\footnote{Whenever discussing universality and ultra-homogeneity for ultrametric spaces, both conditions are restricted to only the collection of ultrametric spaces, e.g., an ultrametric space is universal if it contains isometric copies of all separable ultrametric spaces.} ultrametric space (which we call a \emph{Urysohn universal ultrametric space}), based on which he proved that each separable ultrametric space is isometrically embedable into both $\ell_1$ and $\ell_2$. However, his construction of universal ultrametric space is not separable. In fact, any separable ultrametric space must have a countable spectrum\footnote{The spectrum of a metric space $(X,d_X)$ is the distance set $\spec(X)\coloneqq\{d_X(x,x'):\,x,x'\in X\}$.} (see for example \cite{gao2011polish}) and thus it does not satisfy the universality condition. By restricting to only ultrametric spaces with fixed countable spectrum $R\subset\mathbb{R}_{\geq 0}$, Gao and Shao \cite{gao2011polish} turned to consider the so-called $R$-universality and $R$-ultra-homogeneity conditions and thus defined the $R$-Urysohn universal ultrametric space (cf. Definition \ref{def:R-urysohn}). They provided several constructions and proved uniqueness of the $R$-universal and $R$-ultra-homogeneous ultrametric space for any $R\subset \mathbb{R}_{\geq 0}$ who contains $0$. 

The \emph{Gromov-Hausdorff distance} $\dgh$ introduced by Gromov in \cite{gromov1981groups} is a natural distance comparing compact metric spaces. 
The \emph{Gromov-Hausdorff ultrametric} $\ugh$ was first introduced by Zarichnyi in \cite{zarichnyi2005gromov} as an analogue to $\dgh$ for comparing compact ultrametric spaces. Denote by $\ums$ the collection of all compact ultrametric spaces. Zarichnyi established that $(\ums,\ugh)$ is a complete but not separable ultrametric space. Some theoretical and computational aspects of $(\ums,\ugh)$ have been further studied in \cite{qiu2009geometry,memoli2019gromov}. In particular, a structural theorem for $\ugh$ (cf. Theorem \ref{thm:ugh}) is identified in \cite{memoli2019gromov} which significantly helps  in estimating and computing $\ugh$ throughout this paper.

\paragraph{Contributions.} We establish in this paper the universality and ultra-homogeneity of the Gromov-Hausdorff ultrametric space $(\ums,\ugh)$. This result is interesting in that the collection of all compact ultrametric spaces is itself universal for ultrametric spaces. We then naturally identify a novel construction of the $R$-Urysohn universal ultrametric space for any countable $R\subset\mathbb{R}_{\geq 0}$ containing $0$ using $(\ums,\ugh)$. In the course of proving universality of $(\ums,\ugh)$, we developed a notion named by \emph{admissible order} which has a close relation with graphical representations of dendrograms. This concept allows us to at least prove finite universality of $\ums$, i.e., finite ultrametric spaces can be isometrically embedded into $\ums$. Though in the end we could not yet prove universality of $\ums$ via this approach, we think the concept of an meaningful order on an ultrametric space is interesting itself and we provide a detailed discussion in Section \ref{sec:admissible-order}.

\paragraph{Related work.} It is natural to wonder what is the relation between the Urysohn universal space $\mathbb{U}$ and $(\ms,\dgh)$, the collection of all compact metric spaces endowed with the Gromov-Hausdorff distance. In \cite[Ch. 3, Exercise (b)]{gromov2007metric}, Gromov first observed that $\ms\cong \mathrm{H}(\mathbb{U})/\mathrm{Iso}(\mathbb{U})$, where $\mathrm{H}(\mathbb{U})$ denotes the hyperspace of $\mathbb{U}$ consisting of all nonempty compact subsets of $\mathbb{U}$ endowed with the Hausdorff distance $d_\mathrm{H}^\mathbb{U}$ and $\mathrm{Iso}(\mathbb{U})$ denotes the isometry group of $\mathbb{U}$; see also \cite{antonyan2020gromov} for more details and a proof. This implicitly implies that $\ms$ is not isometric to $\mathbb{U}$. In fact, it was proved later in \cite{iliadis2017local} that $(\ms,\dgh)$ does not satisfy the ultra-homogeneity. However as for universality, the authors in \cite{iliadis2017local} proved that the collection $\ms$ of all compact metric spaces endowed with the Gromov-Hausdorff distance $\dgh$ contains isometric copies of all finite metric spaces. It still remains open whether $(\ms,\dgh)$ is a universal space containing isometric copies of all separable (or even just compact) metric spaces. 

\section{Preliminaries}\label{sec:pre}
{\paragraph{Notions about metric spaces.} A metric space is a pair $(X,d_X)$ where $X$ is a set and $d_X$ is a function $d_X:X\times X\rightarrow[0,\infty)$ satisfying the following three conditions:
\begin{enumerate}
    \item for any $x,y\in X$, $d_X(x,y)\geq 0$ and the equality holds if and only if $x=y$; 
    \item for any $x,y\in X$, $d_X(x,y)=d_X(y,x)$; 
    \item for any $x,y,z\in X$, $d_X(x,z)\leq d_X(x,y)+d_X(y,z).$
\end{enumerate}
We say that two metric spaces $(X,d_X)$ and $(Y,d_Y)$ are \emph{isometric}, denoted by $(X,d_X)\cong (Y,d_Y)$ or simply by $X\cong Y$, if there exists a bijective map $\varphi:X\rightarrow Y$ such that for any $x,x'\in X$, $d_X(x,x')=d_Y(\varphi(y),\varphi(y')). $ We call any such bijective map $\varphi$ an \emph{isometry}. We define the \emph{spectrum} $\spec(X)$ of a metric space $(X,d_X)$ to be the set $\spec(X)\coloneqq\{d_X(x,x'):\,x,x'\in X\}$.}

\paragraph{Ultrametric spaces and dendrograms.} 
A metric space $(X,d_X)$ is called an ultrametric space if $d_X$ satisfies the so-called strong triangle inequality: for any $x,y,z\in X$,
\begin{equation}\label{eq:strong-triangle}
    d_X(x,z)\leq \max(d_X(x,y),d_X(y,z)).
\end{equation}
We usually denote by $u_X$ (instead of $d_X$) the metric of an ultrametric space. 

One important visualization of an ultrametric space is that of a \emph{dendrogram}. A dendrogram is a tree representation of hierarchical clustering of a metric space. See below for a precise definition and Figure \ref{fig:ultra-dendro} for a graphical representation of a dendrogram:

\begin{figure}
    \centering
    \includegraphics{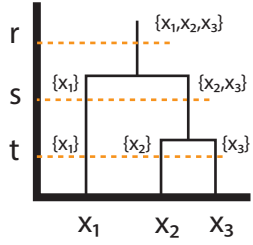}
    \caption{\textbf{Graphical representation of a dendrogram.} In the figure we represent a dendrogram $\theta_X$ over the three-point space $X=\{x_1,x_2,x_3\}$ via a rooted tree (horizontal line segments should be regarded as vertices). For example, $\theta_X(r)=\{\{x_1,x_2,x_3\}\}$, $\theta_X(s)=\{\{x_1\},\{x_2,x_3\}\}$ and $\theta_X(t)=\{\{x_1\},\{x_2\},\{x_3\}\}$.}
    \label{fig:ultra-dendro}
\end{figure}

\begin{definition}[Dendrogram \cite{carlsson2010characterization}]
Given a finite set $X$, a dendrogram is a function $\theta_X:[0,\infty)\rightarrow \mathrm{Part}(X)$\footnote{$\mathrm{Part}(X)$ denotes the set of all partitions of $X$. For each $\{X_1,\cdots,X_k\}\in\mathrm{Part}(X)$, we call an element $X_i$ a block for $i=1,\cdots,k$.} that satisfies the following conditions:
\begin{enumerate}
    \item $\theta_X(0)=\{\{x\}:\,x\in X\}$ is the singleton partition;
    \item for any $0\leq t<s$, $\theta_X(s)$ is coarser than $\theta_X(t)$, i.e., for each block $B\in\theta_X(t)$, there exists a block $C\in\theta_X(s)$ such that $B\subset C$;
    \item there exists $T>0$ such that $\theta_X(T)=\{X\}$;
    \item for each $t\geq 0$, there exists $\eps>0$ such that $\theta_X(s)=\theta_X(t)$ for all $s\in[t,t+\eps]$.
\end{enumerate}
\end{definition}

Given a finite set $X$, denote by $\mathcal{D}(X)$ the collection of all dendrograms over $X$ and by $\ums(X)$ the collection of all ultrametrics on $X$. Then, there exists a bijection between $\mathcal{D}(X)$ and $\ums(X)$. For completeness, we describe a bijective map $\Phi_X:\mathcal{D}(X)\rightarrow\mathcal{U}(X)$ and its inverse $\Psi:\mathcal{U}(X)\rightarrow\mathcal{D}(X)$ as follows:

For any $\theta\in\mathcal{D}(X)$, we define $u:X\times X\rightarrow\mathbb{R}_{\geq 0}$ by $$u(x,x')\coloneqq\inf\{t\geq 0:\, x\text{ and }x'\text{ belong to the same block in }\theta(t)\},\quad\forall x,x'\in X.$$
It is easy to check that $u\in\mathcal{U}(X)$ and we let $\Phi_X(\theta)=u$.

Now for any $u\in\mathcal{U}(X)$, we define $\theta:[0,\infty)\rightarrow \mathrm{Part}(X)$ by $\theta(t)\coloneqq\{[x]_t^X:\,x\in X\}$ for each $t\geq 0$, where $[x]_t^X\coloneqq\{x'\in X:\,u(x,x')\leq t\}.$ Then, $\theta\in \mathcal{D}(X)$ and we let $\Psi_X(u)=\theta$. 

It is easy to check that both $\Phi_X$ and $\Psi_X $ are bijective and they are inverse to each other; see \cite{carlsson2010characterization} for more details. In the sequel we will also use notation $[x]_t=[x]_t^X$ when the underlying set $X$ is clear from the context. From now on, for any given ultrametric space $(X,u_X)$, we always denote by $\theta_X=\Psi_X(u_X)$ the dendrogram corresponding to $u_X$.

\paragraph{Quotient operator on ultrametric spaces.} The following is a crucial operator for ultrametric spaces defined in \cite{memoli2019gromov}. 
\begin{definition}[Quotient]
Given a finite ultrametric space $X$ and $t\geq 0$, we let $X_t\coloneqq\{[x]_t:\,x\in X\}$. We construct an ultrametric $u_{X_t}$ on $X_t$ as follows:
$$u_{X_t}([x]_t,[x']_t)\coloneqq\begin{cases}u_X(x,x'), & u_X(x,x')>t\\ 0, & u_X(x,x')\leq t\end{cases}.$$
We call $(X_t,u_{X_t})$ the quotient of $X$ at level $t$.
\end{definition}

Note that $(X_0,u_{X_0})\cong (X,u_X)$. Intuitively speaking, the dendrogram $\theta_{X_t}$ corresponding to $(X_t,u_{X_t})$ is generated from the one $\theta_X$ corresponding to $(X,u_X)$ by simply forgetting structures of $\theta_X$ below level $t$; see Figure \ref{fig:quotient} for an illustration. It is worth noting that the quotient of a compact ultrametric space is still compact and the quotient of a Polish ultrametric space remains Polish. Furthermore, we have the following two more refined results.

\begin{figure}
    \centering
    \includegraphics{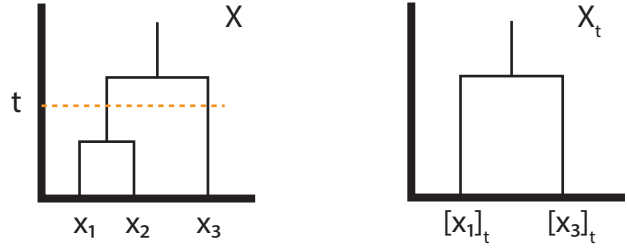}
    \caption{\textbf{Dendrogram illustration of the quotient operator.} The leftmost dendrogram represents of a three-point ultrametric space $X$, whereas the rightmost dendrogram represents the quotient $X_t$ of $X$ at level $t$. Note that after taking the quotient, structures of $\theta_X$ below $t$ disappear.}
    \label{fig:quotient}
\end{figure}

\begin{lemma}\label{lm:compact-quotient}
If $X$ is a compact ultrametric space, then for any $t>0$, $X_t$ is a finite space.
\end{lemma}
\begin{proof}
Since $X$ is compact, there exists a finite $t$-net $X_N\subset X$, i.e., for any $x\in X$, there exists $x_N\in X_N$ such that $u_X(x,x_N)\leq t$. Hence, $[x]_t=[x_N]_t$. Therefore, $X_t\subset\{[x_N]_t:\,x_N\in X_N\}$ and thus $X_t$ is a finite set.
\end{proof}

\begin{lemma}\label{lm:countable-quotient}
If $X$ is a Polish ultrametric space, then for any $t>0$, $X_t$ is a countable space.
\end{lemma}
\begin{proof}
Let $X_c\subset X$ be a countable dense subset. Then, $\{[x_c]_t:\,x_c\in X_c\}$ is a countable subset of $X_t$. For any $x\in X$, there exists $x_c\in X_c$ such that $u_X(x,x_c)\leq t$ since $X_c$ is dense. Hence, $[x]_t=[x_c]_t$. Therefore, $X_t\subset\{[x_c]_t:\,x_c\in X_c\}$ and thus $X_t$ is countable.
\end{proof}

\paragraph{{Definitions of} $\dgh$ and $\ugh$.} Recall that $\ums$ denotes the collection of all compact ultrametric spaces. There exists a natural ultrametric on $\ums$ analogous to the Gromov-Hausdorff distance on $\ms$, the collection of all compact metric spaces. We first briefly review the definition of the Gromov-Hausdorff distance.

\begin{definition}[Gromov-Hausdorff distance]
Given two metric spaces $X$ and $Y$, we define the Gromov-Hausdorff distance $\dgh(X,Y)$ between them as follows:
$$\dgh(X,Y)\coloneqq\inf_Z \dH^Z(X,Y), $$
where the infimum is taken over all metric spaces $Z$ and isometric embeddings from $X$ to $Z$ and from $Y$ to $Z$.
\end{definition}

Now, we modify the definition of $\dgh$ by infimizing over $Z\in\ums$ instead of $Z\in\ms$ to obtain the Gromov-Hausdorff ultrametric:

\begin{definition}[Gromov-Hausdorff ultrametric]
Given two ultrametric spaces $X$ and $Y$, we define the Gromov-Hausdorff ultrametric $\ugh(X,Y)$ between them as follows:
$$\ugh(X,Y)\coloneqq\inf_Z \dH^Z(X,Y), $$
where the infimum is taken over all \emph{ultrametric} spaces $Z$ and isometric embeddings from $X$ to $Z$ and from $Y$ to $Z$.
\end{definition}

The following structural formula provides a precise method for computing $\ugh$ using the quotient operator on ultrametric spaces.

\begin{theorem}[Structural formula for $\ugh$ \cite{memoli2019gromov}]\label{thm:ugh}
Given $X,Y\in\ums$, we have
$$\ugh(X,Y)=\min\{t\geq 0:\,X_t\cong Y_t\}. $$
\end{theorem}

For $\eps\geq 0$, denote $\spec_\eps(X)\coloneqq\{t\in\spec(X):\,t\geq \eps\}$. Then, as an immediate consequence of Theorem \ref{thm:ugh}, we have the following:
\begin{corollary}\label{coro:app-ugh}
Given $X,Y\in\ums$, we have
$$\ugh(X,Y)\geq { \inf}\{\eps\geq 0:\,\spec_\eps(X)=\spec_\eps(Y)\}. $$
\end{corollary}
This corollary was first mentioned in \cite[Theorem 4.2]{qiu2009geometry}. Please also see \cite[Theorem 5.13]{memoli2019gromov} for a generalization.

\section{Urysohn universal ultrametric space}\label{sec:const}

\begin{definition}
Given an ultrametric space $X$, we say $X$ is \emph{universal}, if any Polish ultrametric space is isometrically embedable into $X$; we say $X$ is \emph{ultra-homogeneous} if for any finite ultrametric space $B$, a subset $A$ and an isometric embedding $\varphi:A\rightarrow X$, there exists an isometric extension $\psi:B\rightarrow X$ such that $\psi|_A=\varphi$.
\end{definition}

We call an ultrametric space $X$ an \emph{Urysohn universal ultrametric space} if $X$ satisfies the universality and ultra-homogeneity conditions. We do not require $X$ to be Polish as in the case of Urysohn universal \emph{metric} space since there exists no separable universal ultrametric space. 

\begin{theorem}\label{thm:u-urysohn}
$(\ums,\ugh)$ is a Urysohn universal ultrametric space.
\end{theorem}

Both universality and ultra-homogeneity properties of $\ums$ follows from the following key observation.

\begin{proposition}[One point extension]\label{prop:1p-ext}
Let $X$ be a finite ultrametric space with $|X|>1$ and $\hat{X}\subset X$ be a subspace such that $|\hat{X}|+1=|X|$. Then, if $\varphi:\hat{X}\rightarrow \ums$ is an isometric embedding, there exists an isometric embedding $\psi:X\rightarrow \ums$ such that $\psi|_{\hat{X}}=\varphi$.
\end{proposition}
\begin{proof}
Assume that $X=\{x_1,\cdots,x_{n+1}\}$ and $\hat{X}=\{x_1,\cdots,x_n\}$ where $n\geq 1$. Let $X_{i}\coloneqq\varphi(x_i)$ for all $i=1,\cdots,n$. Let $\delta\coloneqq\min\{u_X(x_i,x_{n+1}):\,i=1,\cdots,n\}>0$. Then, we let
$$M_{n+1}\coloneqq \{i:\,u_X(x_i,x_{n+1})=\delta,i=1,\cdots,n\}= \mathrm{argmin}_{1\leq i\leq n}u_X(x_i,x_{n+1}).$$
For any $k,l\in M_{n+1}$, we have that $u_X(x_k,x_l)\leq \max(u_X(x_k,x_{n+1}),u_X(x_l,x_{n+1}))=\delta$.  Then, $\ugh(X_k,X_l)=u_X(x_k,x_l)\leq \delta$. By Theorem \ref{thm:ugh} we have that $(X_k)_\delta\cong (X_l)_\delta$. Therefore, there exists $Z\in\ums$ such that $Z\cong (X_k)_\delta$ for each $k\in M_{n+1}$. Let { {$N\coloneqq\max\left\{\left|(X_k)_\frac{\delta}{2}\right|:\,k\in M_{n+1}\right\}$.} By Lemma \ref{lm:compact-quotient}, we have that $Z$ is finite and that $N<\infty$.} Then, we define $X_{n+1}\coloneqq Z\cup \{*_1,\cdots,*_{N+1}\}$, where $\ast_i$s are $(N+1)$ distinguished points not belonging to $Z$ and introduce a function $u_{X_{n+1}}:X_{n+1}\times X_{n+1}\rightarrow\mathbb{R}_{\geq 0}$ as follows: we pick an arbitrary point $z_*\in Z$ and for all $i,j=1,\cdots,N+1$, we let
\begin{enumerate}
    \item $u_{X_{n+1}}(z_*,*_i)=\delta$;
    \item $u_{X_{n+1}}(*_i,*_j)=\delta\cdot\delta_{ij}$;
    \item $u_{X_{n+1}}(z,*_i)=u_{Z}(z_*,z)$, for any $z\in Z\backslash\{z_*\}$;
    \item $u_{X_{n+1}}|_{Z\times Z}=u_Z$.
\end{enumerate}

Then,  it is easy to check that $u_{X_{n+1}}$ is an ultrametric and thus $(X_{n+1},u_{X_{n+1}})\in\ums$. Obviously, we have that $(X_{n+1})_\delta\cong Z\cong(X_k)_\delta$ and thus $\ugh(X_{n+1},X_{k})\leq \delta$ for each $k\in M_{n+1}$. For the opposite inequality, we know from Lemma \ref{lm:compact-quotient} that both $(X_{k})_t$ and $(X_{n+1})_t$ are finite spaces for {  $\frac{\delta}{2}<t<\delta$}. By counting cardinalities, we have
$$|(X_k)_t|\leq { \left|(X_k)_\frac{\delta}{2}\right|}\leq N<N+1+|Z|=|(X_{n+1})_t|.$$
This implies that $(X_{n+1})_t\not\cong(X_{k})_t$. Therefore, $\ugh(X_{n+1},X_{k})= \delta=u_X(x_{n+1},x_k)$. 

Now for any $j\notin M_{n+1}$, we have that $u_X(x_j,x_{n+1})>\delta$. For each $k\in M_{n+1}$, we apply the strong triangle inequality for points $x_j,x_{n+1}$ and $x_k$. Then, we must have that
$$u_X(x_j,x_k)=u_X(x_j,x_{n+1})>\delta=u_X(x_k,x_{n+1}).$$
Thus, 
$$\ugh(X_j,X_k)=u_X(x_j,x_k)>u_X(x_k,x_{n+1})=\ugh(X_k,X_{n+1}).$$
Then, by applying the strong triangle inequality for $X_j,X_{n+1}$ and $X_k$, we have that 
$$\ugh(X_j,X_{n+1})=\ugh(X_j,X_k)=u_X(x_j,x_k)=u_X(x_j,x_{n+1}).$$
Therefore, the map $\psi:X\rightarrow\ums$ taking $x_i$ to $X_i$ for all $i=1,\cdots,n$ and $x_{n+1}$ to $X_{n+1}$ is an isometric embedding such that $\psi|_{\hat{X}}=\varphi$. 
\end{proof}

Though the proof of Proposition \ref{prop:1p-ext} is long, the essential idea and the constructions are easy to understand via dendrograms. See Figure \ref{fig:one-point-ext} for an illustration of the proof of Proposition \ref{prop:1p-ext} using dendrogram representations.

\begin{figure}[ht]
    \centering
    \includegraphics[width=\linewidth]{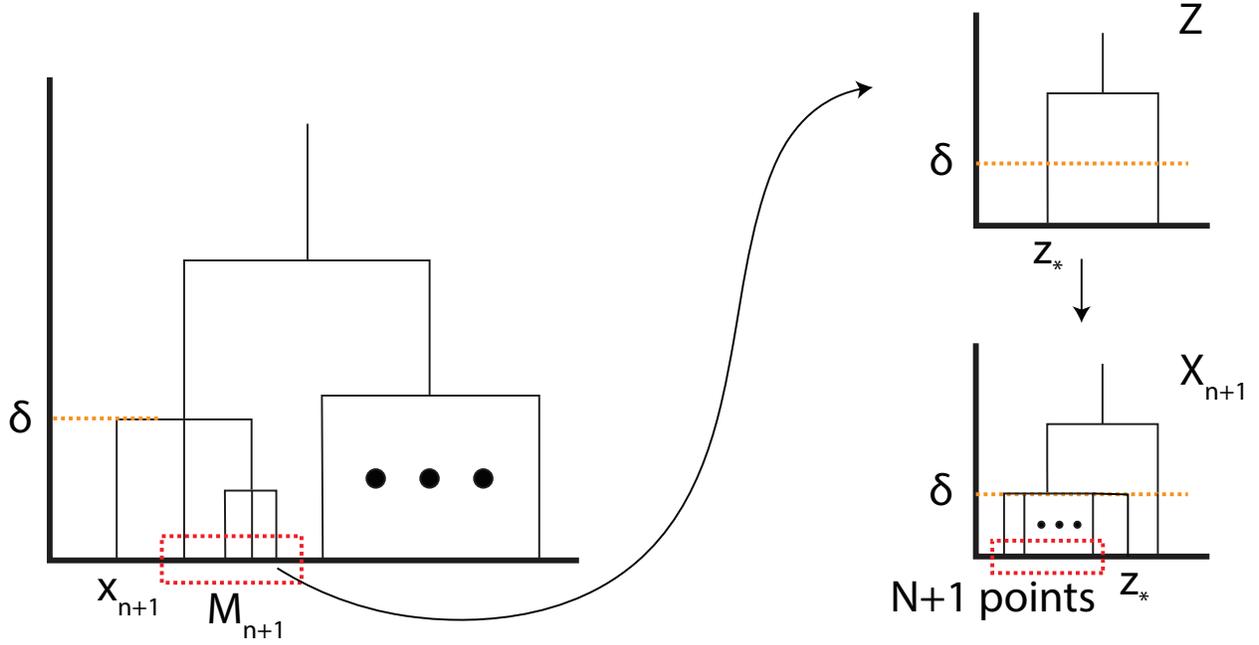}
    \caption{\textbf{Illustration of the proof of Proposition \ref{prop:1p-ext}.}}
    \label{fig:one-point-ext}
\end{figure}

\begin{proof}[Proof of Theorem \ref{thm:u-urysohn}]
We first prove universality of $\ums$. Assume that $X$ is a Polish ultrametric space. Let $X_c$ be a countable dense subset of $X$. Assume that $X_c=\{x_1,x_2,\cdots\}$ and let $X_n\coloneqq\{x_1,\cdots,x_n\}$ for $n=1,\cdots$. We construct an arbitrary map $\varphi_1:X_1\rightarrow\ums$. Then, by Proposition \ref{prop:1p-ext}, we inductively construct isometric embeddings $\varphi_n:X_n\rightarrow \ums$ for all $n=1,\cdots$ such that  $\varphi_{n+1}|_{X_n}=\varphi_n$. We then define a map $\varphi:X_c\rightarrow\ums$ as follows: for any $x\in X_c$, there exists $n$ such that $x\in X_n$ and we let $\varphi(x)\coloneqq\varphi_n(x)$. It is easy to see that $\varphi$ is well-defined and it is an isometric embedding. Now, since $X_c$ is dense in $X$ and $\ums$ is complete, there exists an extension $\hat{\varphi}:X\rightarrow\ums$ of $\varphi$ which is still an isometric embedding. Therefore, $\ums$ is universal.

Now we prove that $\ums$ is ultra-homogeneous. Suppose we have a finite ultrametric space $B$, a subset $A$ and an isometric embedding $\phi:A\rightarrow \ums$. Assume without loss of generality that $B=\{x_1,\cdots,x_n\}$ and $A=\{x_1,\cdots,x_k\}$ where $n\geq 2$ and $1\leq k< n$. Let $A_j\coloneqq A\cup\{x_{k+1},\cdots,x_{k+j}\}$ for $j=1,\cdots,n-k$. Let $A_0\coloneqq A$ and $\phi_0\coloneqq \phi$. Then, by Proposition \ref{prop:1p-ext} again, there exist isometric embeddings $\phi_j:A_j\rightarrow \ums$ for each $j=0,\cdots,n-k$ such that $\phi_{j+1}|_{A_j}=\phi_j:A_j\rightarrow\ums$. Therefore, $\psi\coloneqq \phi_{n-k}:A_{n-k}=B\rightarrow\ums$ is an isometric embedding such that $\psi|_A=\phi_{n-k}|_{A_0}=\phi_0=\phi$ and thus $\ums$ is ultra-homogeneous.
\end{proof}

\begin{remark}\label{rmk:spec}
In the proof of universality above, we can choose $\varphi_1$ such that $\varphi_1(x_1)=*$, where $\ast$ denotes the one-point ultrametric space. Then, by the construction in the proof of Proposition \ref{prop:1p-ext} and induction, we have that $\spec(\varphi(x_n))\subset\spec(X_n)$ for $n=1,\cdots$. Applying Corollary \ref{coro:app-ugh}, it is easy to see that $\hat{\varphi}:X\rightarrow\ums$ satisfying the following: for any $x\in X$, $\spec(\hat{\varphi}(x))\subset \spec(X)$. Similarly, in the proof of ultra-homogeneity above, if $\phi:A\rightarrow \ums$ is chosen such that $\spec(\phi(a))\subset\spec(B)$ for each $a\in A$, then $\psi:B\rightarrow\ums$ also satisfies that $\spec(\phi(b))\subset\spec(B)$ for each $b\in B$.
\end{remark}

\subsection{$R$-Urysohn universal ultrametric space} 
Let $R\subset\mathbb{R}_{\geq 0}$ be a countable set containing 0. We call an ultrametric space $X$ an \emph{$R$-ultrametric space} if $\spec(X)\subset R$.

\begin{definition}[\cite{gao2011polish}]\label{def:R-urysohn}
An $R$-ultrametric space $(X,u_X)$ is called an \emph{$R$-Urysohn universal ultrametric space} if 
\begin{enumerate}
    \item $X$ is Polish;
    \item $X$ is \emph{$R$-universal}, i.e., $X$ contains isometric copies of all Polish $R$-ultrametric spaces;
    \item $X$ is \emph{$R$-ultra-homogeneous}, i.e., $X$ satisfies the ultra-homogeneous condition for all finite $R$-ultrametric spaces.
\end{enumerate}
\end{definition}

According to \cite{gao2011polish}, for any countable $R\subset\mathbb{R}_{\geq 0}$, the $R$-Urysohn universal ultrametric space is unique up to isometry.

Denote by $\ums_R$ the collection of all compact ultrametric spaces $(X,u_X)$ such that $\spec(X)\subset R$. We still denote by $\ugh$ its restriction to $\ums_R$.

\begin{theorem}
$(\ums_R,\ugh)$ is the $R$-Urysohn universal ultrametric space.
\end{theorem}

\begin{proof}
Both $R$-universality and $R$-ultra-homogeneity of $(\ums_R,\ugh)$ follows from the same proof of Theorem \ref{thm:u-urysohn} combined with Remark \ref{rmk:spec}. 

Now we only need to show that $(\ums_R,\ugh)$ is a Polish space. Since $\ums$ is complete, any Cauchy sequence $\{X_n\}$ in $\mathcal{U}_R$ has a limit $X$ in $\ums$. Denote $\delta_n\coloneqq\ugh(X_n,X)$. If $\delta_n=0$ for some $n\in\mathbb{N}$, then $X\cong X_n$ and thus $X\in\ums_R$. Now we assume $\delta_n>0$ for all $n\in\mathbb{N}$. Then, by Corollary \ref{coro:app-ugh}, we have that
$$\ugh(X_n,X)\geq { \inf}\{\eps\geq 0:\,\spec_\eps(X)=\spec_\eps(X_n)\}. $$
Then, $\spec_{ {2\delta_n}}(X)=\spec_{ {2\delta_n}}(X_n)\subset R$. For any $0<t\in \spec(X)$, there exists $n$ large enough such that $t>{ {2\delta_n}}$ and thus $t\in\spec_{ {2\delta_n}}(X)\subset R$. Therefore, $\spec(X)\subset R$ and thus $X\in \ums_R$. This implies that $\ums_R$ is complete. Furthermore, the set of all finite ultrametric spaces with spectrum contained in $R$ is a countable dense subset of $\ums_R$ and thus $\ums_R$ is separable. In conclusion, $\ums_R$ is a Polish metric space.
\end{proof}

\subsection{Admissible order and universality}\label{sec:admissible-order}
In this section, we establish that each Polish ultrametric space admits a special total order and discuss one possible alternative approach to prove universality of $\ums$.

\begin{definition}\label{def:adm-order}
For any ultrametric space $(X,u_X)$, we call a total order $\leq$ on $X$ an \emph{admissible order} if $u_X(x,y)\leq u_X(x,z)$ whenever $x,y,z\in X$ and $x< y< z$.
\end{definition}

Assume that $X=\{x_1,\cdots,x_n\}$ endowed with $u_X$ is a finite ultrametric space. Suppose that $\leq$ is a total order on $X$ and assume without loss of generality that $x_i<x_j$ if and only if $i<j$. Then, by placing $x_i$s along the real line according to the order $x_1<x_2<\cdots<x_n$, one can draw a graphical representation of the dendrogram $\theta_X$ corresponding to $u_X$ without self-crossing if and only if the total order $\leq$ is admissible. See Figure \ref{fig:ultra-order} for an illustration when $n=3$. 

On the other hand, each graphical representation of $\theta_X$ without self-crossings (which must exists since a tree is a planar graph) gives rise to an admissible order on $X$: $x_i<x_j$ if $x_i$ is on the left of $x_j$. This actually implies that each finite ultrametric space admits at least on admissible order. We provide an alternative formal proof of the fact in the following lemma.

\begin{lemma}\label{lm:order-finite}
Assume that $X=\{x_1,\cdots,x_n\}$ is a finite ultrametric space. Then, $X$ admits an admissible order.
\end{lemma}
\begin{proof}
We prove by induction on $n$. The case when $n=1$ is trivial. Now assume that $n>1$ and the claim holds true for all $(n-1)$-point ultrametric spaces. Then, there exists a total order on $X\backslash\{x_n\}$ such that $u_X(x_k,x_i)\leq u_X(x_k,x_j)$ for all $x_k< x_i< x_j$ and $1\leq k,i,j<n$. Let $M_n\coloneqq\mathrm{argmin}_{1\leq i<n}u_X(x_i,x_n)$. Take ${i_0}\in M_n$ such that $x_{i_0}\leq x_i$ for each $i\in M_n$. Then, for each $1\leq i<n$, if $x_i\geq x_{i_0}$, we let $x_n<x_i$; otherwise, we let $x_n>x_i$. This assignment gives rise to a total order on $X$. Now we check that this total order is admissible. It suffices to check the conditions for admissible orders for three distinct points $x_i,x_j,x_n\in X$ in the following three cases:
\begin{enumerate}
    \item $x_n<x_i<x_j$: by construction of the total order, we have that $x_{i_0}\leq x_i<x_j$. If $x_{i_0}=x_i$, then by definition of $i_0$, we have that $u_X(x_n,x_i)=u_X(x_n,x_{i_0})\leq u_X(x_n,x_j)$. Now we assume that $x_{i_0}<x_i$. Since $\leq$ is admissible on $X\backslash\{x_n\}$, we have that $u_X(x_{i_0},x_i)\leq u_X(x_{i_0},x_j)$. Therefore,
    \begin{align*}
        u_X(x_n,x_i)\leq \max(u_X(x_n,x_{i_0}),u_X(x_{i_0},x_i))\leq \max(u_X(x_n,x_{j}),u_X(x_{i_0},x_j)).
    \end{align*}
    Since $u_X(x_j,x_{i_0})\leq\max(u_X(x_j,x_n),u_X(x_n,x_{i_0}))\leq u_X(x_j,x_n)$, we have that $u_X(x_n,x_i)\leq u_X(x_n,x_{j}).$
    \item $x_i<x_n<x_j$: in this case we have that $x_i< x_{i_0}\leq x_j$. Then, $u_X(x_i,x_n)>u_X(x_n,x_{i_0})$ since $i\notin M_n$. Thus $u_X(x_i,x_n)=u_X(x_i,x_{i_0})$ by the strong triangle inequality. Therefore, $u_X(x_i,x_n)=u_X(x_i,x_{i_0})\leq u_X(x_i,x_j)$ since $\leq$ is admissible on $X\backslash\{x_n\}$.
    \item $x_i<x_j<x_n$: similar to the second case, we have that $x_i<x_j<x_{i_0}$ and $u_X(x_i,x_n)=u_X(x_i,x_{i_0})$. Therefore, $u_X(x_i,x_j)\leq u_X(x_i,x_{i_0})= u_X(x_i,x_n)$.
\end{enumerate}
\end{proof}

\begin{figure}
    \centering
    \includegraphics{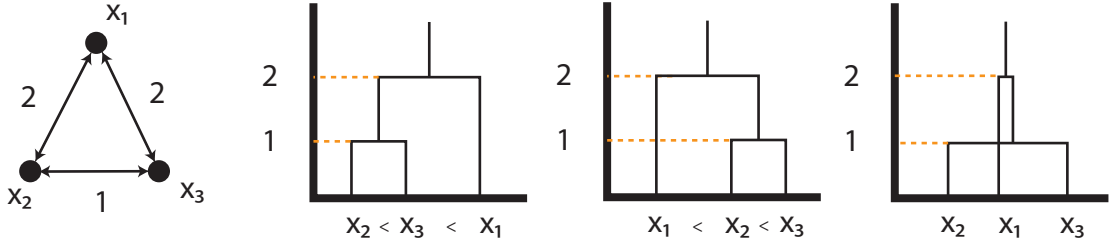}
    \caption{\textbf{Relation between admissible orders and dendrogram representations.} The leftmost figure is a 3-point ultrametric space $X$. The rest are graphical representations of the dendrogram $\theta_X$ based on three different total orders on $X$. Note that the middle two total orders are admissible and the corresponding graphical representations have no self-crossing whereas the rightmost order $x_2<x_1<x_3$ is not admissble and results in a self-crossing in the corresponding graphical representation of $\theta_X$. On the contrary, the middle two figures can be interpreted in the sense that each planar representation of $\theta_X$ generates an admissible order on the ultrametric space.}
    \label{fig:ultra-order}
\end{figure}

Based on Lemma \ref{lm:order-finite}, one can easily prove universality of $\ums$ for finite spaces:

\begin{proposition}\label{prop:finite-embedding}
Any finite ultrametric space is isometrically embedable in $(\ums,\ugh)$.
\end{proposition}

\begin{proof}
Assume $X=(\{x_1,\cdots,x_n\},u_X)$ is a finite ultrametric space. By Lemma \ref{lm:order-finite}, $X$ admits an admissible order. Relabel $X$ such that $x_i<x_j$ if and only if $i<j$. Define $f:X\rightarrow \ums$ by mapping $x_i$ to $X_i\coloneqq\{x_j:j\leq i\}$ endowed with the restricted metric $u_X|_{X_i\times X_i}$. Take any $1\leq i<j\leq n$ and let $t\coloneqq u_X(x_i,x_j)>0$. Then, for each $x_k\in X_j\backslash X_i$, we have $u_X(x_i,x_k)\leq u_X(x_i,x_j)=t$ since $x_i<x_k\leq x_j$. Therefore, $[x_k]_t^{X_j}=[x_i]_t^{X_j}$ in $(X_j)_t$ for each $x_k\in X_j\backslash X_i$. This implies that the map $\varphi_t:(X_i)_t\rightarrow (X_j)_t$ taking $[x_l]_t^{X_i}$ to $[x_l]_t^{X_j}$ for each $l=1,\cdots,i$ is bijective and thus an isometry. By Theorem \ref{thm:ugh}, we have that $\ugh(X_i,X_j)\leq t$.

Now for any $s<t$, we have that $[x_j]_s^{X_j}\neq[x_i]_s^{X_j}$. Then, $|(X_j)_s|\geq|(X_i)_s|+1$ and thus $(X_i)_s\not\cong (X_j)_s$. Therefore, by Theorem \ref{thm:ugh} again we have that $\ugh(X_i,X_j)=t=u_X(x_i,x_j)$ and thus $f:X\rightarrow\ums$ taking $x_i$ to $X_i$ for each $i=1,\cdots,n$ is an isometric embedding.
\end{proof}

Now we generalize Lemma \ref{lm:order-finite} to the case of Polish ultrametric spaces.

\begin{lemma}\label{lm:order-c}
Any countable ultrametric space admits an admissible order.
\end{lemma}
\begin{proof}
The proof is essentially the same as the one for Lemma \ref{lm:order-finite}. Assume $X=(\{x_1,\cdots\},u_X)$ is a countable ultrametric space. We construct an admissible order on $X$ by extending an admissible order on $X_n\coloneqq\{x_1,\cdots,x_n\}$ inductively for $n=1,2,\cdots$. 

Obviously, when $n=1$, there is nothing to construct. Assume that $n\geq 1$ and we have introduced an admissible order on $X_n$. Then, $u_X(x_k,x_i)\leq u_X(x_k,x_j)$ for all $x_k< x_i< x_j$ and $1\leq k,i,j\leq n$. Let $M_{n+1}\coloneqq\mathrm{argmin}_{1\leq i\leq n}u_X(x_i,x_{n+1})$. Take ${i_0}\in M_{n+1}$ such that $x_{i_0}\leq x_i$ for each $i\in M_{n+1}$. Then, for each $1\leq i\leq n$, if $x_i\geq x_{i_0}$, we let $x_{n+1}<x_i$; otherwise, we let $x_{n+1}>x_i$. It is easy to check as in the proof of Lemma \ref{lm:order-finite} that this assignment gives rise to an admissible order on $X_{n+1}$. Therefore, by this inductive process, there will be an admissible order on $X$.
\end{proof}

\begin{theorem}\label{thm:order-polish}
Assume that $X$ is a Polish ultrametric space. Then, $X$ admits an admissible order.
\end{theorem}
\begin{proof}
For each $n\in\mathbb{N}$, we construct inductively an admissible order $\leq_n$ on the quotient space $X_\frac{1}{n}$ such that if $[x]_\frac{1}{n}<_n[x']_\frac{1}{n}$ then $[x]_\frac{1}{m}<_m[x']_\frac{1}{m}$ for each pair $m>n$: by Lemma \ref{lm:countable-quotient}, $X_1$ is a countable set and thus by Lemma \ref{lm:order-c}, $X_1$ admits an admissible order $\leq_1$. Now suppose we have identified an admissible order $\leq_k$ on $X_\frac{1}{k}$ for each $1\leq k\leq n$ such that $[x]_\frac{1}{k}<_k[x']_\frac{1}{k}$ induces $[x]_\frac{1}{l}<_l[x']_\frac{1}{l}$ for all $1\leq k\leq l$. Now, we define a total order $\leq_{n+1}$ on the quotient space $X_{\frac{1}{n+1}}$:
\begin{enumerate}
    \item For any { $[x]_\frac{1}{n}\neq[x']_\frac{1}{n}$}, suppose without loss of generality that $[x]_\frac{1}{n}<_n[x']_\frac{1}{n}$, then we let $[x]_\frac{1}{n+1}<_{n+1}[x']_\frac{1}{n+1}.$
    
    \item Now each $[x]_\frac{1}{n}\in X_\frac{1}{n}$ is partitioned into at most countably many blocks at level $\frac{1}{n+1}$ (where we use Lemma \ref{lm:countable-quotient} again): $[x]_\frac{1}{n}=\cup_{i=1}^\infty[x_i]_{\frac{1}{n+1}}$. We use Lemma \ref{lm:order-c} again to introduce an admissible order $\leq_x$ on the subspace $\left\{[x_i]_\frac{1}{n+1}:i=1,\cdots\right\}\subset X_\frac{1}{n+1}$. Then, we let $[x_i]_\frac{1}{n+1}<_{n+1}[x_j]_\frac{1}{n+1}$ if and only if $[x_i]_\frac{1}{n+1}<_x[x_j]_\frac{1}{n+1}$ for $i,j=1,\cdots$.
\end{enumerate}
It is easy to check that $\leq_{n+1}$ is an admissible order on $X_\frac{1}{n+1}$.

Now we define a total order $\leq$ on $X$: for any $x\neq x'\in X$, suppose $n\in\mathbb{N}$ is such that $[x]_\frac{1}{n}\neq [x']_\frac{1}{n}$. Suppose without loss of generality that { $[x]_\frac{1}{n}<_n [x']_\frac{1}{n}$}, then we let $x<x'$. It is easy to see that this assignment does not depend on the choice of $n\in\mathbb{N}$ and $\leq$ is admissible.
\end{proof}

\paragraph*{Discussion.} Though it seems promising to use Theorem \ref{thm:order-polish} to prove universality of $\ums$ as we do in Proposition \ref{prop:finite-embedding}, there is a technical issue that we have not yet resolved: in the proof of Proposition \ref{prop:finite-embedding}, we use the difference between cardinalities of sets to distinguish non-isometric spaces. However in the case of Polish spaces, the cardinality of any encountered set may be infinite and thus the cardinality comparison method may fail to work. It seems interesting to study whether one could utilize Theorem \ref{thm:order-polish} via some refined analysis to prove the universality of $\ums$.

\subsection*{Acknowledgements} 
This work was partially supported by the NSF through grants CCF-1740761 and DMS-1723003.

\bibliography{biblio-ughUrysohn}
\bibliographystyle{alpha}


\end{document}